\documentclass[10pt,a4paper]{article}
\usepackage[top=3cm, bottom=3cm, left=3cm, right=3cm]{geometry}
\usepackage[latin1]{inputenc}
\usepackage{amsmath}
\usepackage{amsthm}
\usepackage{amsfonts}
\usepackage{amssymb}
\usepackage{graphics}
\usepackage{tabu}
\usepackage{breqn}
\usepackage[hang,flushmargin]{footmisc} 

\usepackage{amsmath,amsthm,amssymb,graphicx, multicol, array}
\usepackage{enumerate}
\usepackage{enumitem}
\usepackage{hyperref}
\usepackage{setspace}
\usepackage{xcolor}
\usepackage{currfile}

\newtheorem{thm}{Theorem}[section]
\newtheorem{lem}{Lemma}[section]

\newtheorem{exa}{Example}[section]

\newtheorem{dfn}{Definition}[section]

\newcommand{\N}{\mathbb{N}}

\newcommand{\R}{\mathbb{R}}
\newcommand{\C}{\mathbb{C}}

\usepackage{fancyhdr}
\title{Counting Primes Rationally And Irrationally}
\date{}
\author{N. A. Carella}

\begin{document}
\thispagestyle{empty}
\date{}
\maketitle

\vskip .25 in

\textbf{\textit{Abstract}:} The recent technique for estimating lower bounds of the prime counting function 
$\pi(x)=\#\{p \leq x: p\text{ prime}\}$ by means of the irrationality measures $\mu(\zeta(s)) \geq 2$ of special values of the zeta function claims 
that $\pi(x) \gg \log \log x/\log \log \log x$. This note improves the lower bound to $\pi(x) \gg \log x$, and extends the analysis to the irrationality measures $\mu(\zeta(s)) \geq 1$ for rational ratios of zeta functions.
\let\thefootnote\relax\footnote{ \today \date{} 
}

\vskip .25 in 


\section{Introduction} \label{s1111}
Let $x\geq 1$ be a large number and define the prime counting function by

\begin{equation} \label{eq1111.99}
\pi(x)= \#\{p\leq x: p\text{ is prime} \}.
\end{equation}

Dozens of proofs are known for estimating a lower bound  of the prime counting function $\pi(x)= \#\{p\leq x: p\text{ is prime} \}$, 
see \cite[pp.\ 3-11]{RP96}, \cite{NW00}. The oldest Euclidean technique has dozens of versions and refinements, and remains a research topic today, see \cite{BA16} and Section \ref{s2299} for some information. The recent techniques introduced in \cite{BW07} and \cite{KE07} for estimating lower bounds of the prime counting function 
$\pi(x)=\#\{p \leq x: p\text{ prime}\}$ by means of the irrationality measures $\mu(\zeta(s)) \geq 2$ of special values of the zeta function claims 
that $\pi(x) \gg \log \log x/\log \log \log x$. This note improves the lower bound to $\pi(x) \gg \log x$, and extends the analysis to the irrationality measures $\mu(\zeta(s)) \geq 1$ for rational ratio of zeta functions. The proofs are independent of the irrationality measures $\mu(\alpha) \geq 1$ of the real numbers $\alpha \in \R$. 

\section{For Irrational Values $\zeta(s)$} \label{s2211}

\begin{thm} \label{thm2211.11} Let $x\geq 2$ be a large number, and 
let $\pi(x)= \#\{p\leq x: p\text{ is prime} \}$. 
Then

\begin{equation}
 \pi(x)\geq c_1\log  x +c_0,
\end{equation}

where $c_0$ and $c_1>0$ are constants.
\end{thm}

\begin{proof} Fix an even number $s > 1$. Let $x \gg 1$ be a large number, and let 
\begin{equation} \label{eq2211.13}
\frac{1}{\zeta(s)}=\prod_{p \geq 2}\left ( 1 -\frac{1}{p^s} \right ) \qquad \qquad \text{ and } \qquad \qquad 
\frac{p_x}{q_x}=\prod_{p \leq x}\left ( 1 -\frac{1}{p^s} \right ), 
\end{equation}
where $\gcd(p_x,q_x)=1$. Then
\begin{equation} \label{eq2211.14}
\left |\frac{1}{\zeta(s)}-\frac{p_x}{q_x} \right | =\frac{1}{\zeta(s)}\left (1-\prod_{p >f(x)}\left ( 1 -\frac{1}{p^s} \right )^{-1} \right ),
\end{equation}
where $f(x)= x$. By Lemma \ref{lem9949.82}, the difference has the upper bound 

\begin{eqnarray} \label{eq2211.15}
\left |\frac{1}{\zeta(s)}-\frac{p_x}{q_x} \right | &=&\frac{1}{\zeta(s)}\left (1-\prod_{p >f(x)}\left ( 1 -\frac{1}{p^s} \right )^{-1} \right ) \\
&\leq &\frac{1}{(s-1)\zeta(s)} \frac{1}{f(x)^{s-1}}    \nonumber.
\end{eqnarray}

Let $\mu=\mu(\zeta(s))\geq 2$ be the irrationality measure of the number $\zeta(s)$, see Definition \ref{dfn9900.01}. Then, applying Lemma \ref{lem1300.82} yields
\begin{equation} \label{eq2211.33}
\frac{1}{q_x^{\mu+\varepsilon}}  \leq \frac{1}{2^{(\pi(x)-s-1)(\mu+\varepsilon)}} \leq \left |\frac{1}{\zeta(s)}-\frac{p_x}{q_x} \right |, 
\end{equation}
where $\varepsilon >0$ is an arbitrary small constant. Comparing \eqref{eq2211.15} and \eqref{eq2211.33} yield
\begin{equation} \label{eq2211.35}
 \frac{1}{2^{(\pi(x)-s-1)(\mu+\varepsilon)}} \leq  \frac{1}{(s-1)\zeta(s)} \frac{1}{f(x)^{s-1}}. 
\end{equation}
Solving for $\pi(x)$ yields
\begin{equation} \label{eq2211.37}
\frac{s-1}{(\mu+\varepsilon)\log 2 } \log f(x) +s+1+\frac{\log (s-1)\zeta(s)}{(\mu+\varepsilon)\log 2 }=c_1\log f(x) +c_0\leq \pi(x) 
\end{equation}
where $c_0$, and $c_1=(s-1)/(\mu+\varepsilon)\log 2>0$ are constants. 
\end{proof}
Significant improvement of the lower bound $\pi(x) \gg \log x$ to $\pi(x) \gg x/\log x$ can be achieved by setting $f(x)=2^{cx/\log x+o(x/\log x)}$, for some constant $c>0$. In fact, this can be viewed as a near proof of the Prime Number Theorem by elementary methods. The earlier related analysis appear almost simultaneously in \cite{BW07},  and \cite{KE07}. 
\section{For Rational Values $\zeta(s)/\zeta(t)$} 
\label{s2212}
\begin{thm} \label{thm111.11} Let $x\geq 2$ be a large number, and let $\pi(x)= \#\{p\leq x: p\text{ is prime} \}$. Then

\begin{equation}
 \pi(x)\geq c_3\log x +c_2,
\end{equation}
where $c_2$ and $c_3>0$ are constants.
\end{thm}

\begin{proof} Let $x \gg 1$ be a large number, and let 
\begin{equation} \label{eq2212.13}
\frac{2}{5}=\prod_{p \geq 2} \frac{p^2-1}{p^2+1} \qquad \qquad \text{ and }\qquad \qquad \frac{p_x}{q_x}=\prod_{p \leq x} \frac{p^2-1}{p^2+1} , 
\end{equation}
where $\gcd(p_x,q_x)=1$. Then
\begin{equation} \label{eq2212.14}
\left |\frac{2}{5}-\frac{p_x}{q_x} \right | =\frac{2}{5} \left (1-\prod_{p > f(x)}\left ( \frac{p^2-1}{p^2+1} \right )^{-1} \right ) ,
\end{equation}
where $f(x)= x$. By the Euclidean Prime Number Theorem, there are infinitely many primes, so the product 
\begin{equation} \label{eq2211.30}
\prod_{p >f(x)}\left ( \frac{p^2-1}{p^2+1} \right )^{-1} \ne0.
\end{equation}
Applying Lemma \ref{lem1200.21} yields 
\begin{eqnarray} \label{eq2212.15}
\left |\frac{2}{5}-\frac{p_x}{q_x} \right | &=&\frac{2}{5} \left (1-\prod_{p >f(x)}\left ( \frac{p^2-1}{p^2+1} \right )^{-1} \right ) \\
&\leq &\frac{c_4}{f(x)}  \nonumber,
\end{eqnarray}
where $c_4>0$ is a constant.
Let $\mu=\mu(5/2)= 1$ be the irrationality measure of the number $5/2$, see Definition \ref{dfn9900.01}. Then, applying Lemma \ref{lem1300.84} yields
\begin{equation} \label{eq2212.33}
\frac{c_5}{q_x^{\mu+\varepsilon}}  \leq \frac{1}{2^{(\pi(x)-s-1)(\mu+\varepsilon)}} \leq \left |\frac{2}{5}-\frac{p_x}{q_x} \right |, 
\end{equation}
where $c_5>0$ is a constant. Comparing \eqref{eq2212.15} and \eqref{eq2212.33} yield
\begin{equation} \label{eq2211.35}
\frac{1}{2^{(\pi(x)-s-1)(\mu+\varepsilon )}} \leq   \frac{c_4}{f(x)^{}}. 
\end{equation}
Solving for $\pi(x)$ yields
\begin{equation} \label{eq2212.37}
\frac{1}{(\mu+\varepsilon)\log 2 } \log f(x) +s+1-\frac{\log c_4}{(1+\varepsilon)\log 2}=c_3\log f(x) +c_2\leq \pi(x) 
\end{equation}
where $c_2$, and $c_3=1/(\mu+\varepsilon)\log 2>0$ are constants. 

\end{proof}


\section{Infinite Products}\label{s1200}
The evaluations of some infinite prime products have rational values. The best known cases are generated by some ratios of zeta functions, and by some ratios of $L$-functions. The zeta function and the $L$-function are defined by
\begin{equation}
\zeta(s) =\sum_{n \geq 1} \frac{1}{n^s}=\prod_{p \geq 2}\left ( 1-\frac{1}{p^{s}} \right ) ^{-1},
\end{equation}
and
\begin{equation}
L(s, \chi) =\sum_{n \geq 1} \frac{\chi(n)}{n^s}=\prod_{p \geq 2}\left ( 1 -\frac{\chi(p)}{p^{s}} \right ) ^{-1},
\end{equation}
where $s \in \C$ is a complex variable, and $\chi$ is a character modulo $q\geq 1$, respectively.
\begin{lem} \label{lem1200.21} For any integer $n \geq1$, the following primes products are rational numbers. 

\begin{enumerate} [font=\normalfont, label=(\roman*)]
\item  For any integer $n \geq1$, 
\begin{equation}
\frac{\zeta(2n)^2}{\zeta(4n)} =\prod_{p \geq 2}\frac{\left (1-p^{-2n}\right ) ^{-2}}{\left (1-p^{-2n}\right ) ^{-1}}=\prod_{p \geq 2}\left ( 1 +\frac{2}{p^{2n}-1} \right ) .
\end{equation}
\item For any odd integers $m,n \geq1$, and a character $\chi$ modulo $q$,
\begin{equation}
\frac{L(m,\chi )^n}{L(mn,\chi)} =\prod_{p \geq 2}\frac{\left (1-\chi(p)p^{-m}\right ) ^{-n}}{\left (1-\chi(p)p^{-mn}\right ) ^{-1}}.
\end{equation}
\end{enumerate}
\end{lem}
\begin{proof} (i) Let $B_{2n}$ be the $2n$th Bernoulli number. Then, zeta ratio has a rational value
\begin{eqnarray} \label{eq1200.53}
\frac{\zeta(2n)^2}{\zeta(4n)} 
&=&\left( \frac{(-1)^{n-1}(2 \pi)^{2n} B_{2n}}{2(2n)!}\right ) ^{2}
\left( \frac{(-1)^{2n-1}2(4n)!}{(2 \pi)^{4n} B_{4n}}\right ) \nonumber\\
&=&\frac{ (4n)!B_{2n}^2}{2((2n)!)^2 B_{4n}}\nonumber
\end{eqnarray}
since any $B_k$ is rational. On the other hand, the infinite product has the expression
\begin{eqnarray} \label{eq1200.55}
\frac{\zeta(2n)^2}{\zeta(4n)}  
&=&\prod_{p \geq 2}\frac{\left (1-p^{-2n}\right ) ^{-2}}{\left (1-p^{-4n}\right ) ^{-1}} \nonumber\\
&=&\prod_{p \geq 2}\left ( 1 +\frac{2}{p^{2n}-1} \right ) \nonumber.
\end{eqnarray}
as claimed. (ii) The proof of this case has similar calculations but lengthier.
\end{proof}
\begin{exa}  {\normalfont 
\begin{enumerate} [font=\normalfont, label=(\arabic*)] A pair of rational primes products.
\item  The simplest case occurs for $n =2$. The zeta values are $\zeta(2)=\pi^2/6$ and $\zeta(4)=\pi^4/90$. Thus,
\begin{equation} \label{eq1200.55}
\frac{5}{2} =\frac{\zeta(2)^2}{\zeta(4)} 
=\prod_{p \geq 2}\frac{\left (1-p^{-2}\right ) ^{-2}}{\left (1-p^{-4}\right ) ^{-1}} \prod_{p \geq 2}\frac{p^2+1}{p^2-1}.
\end{equation}
\item The simplest case occurs for $m =1$, $n =3$,  and the quadratic character $\chi$ modulo $q=4$. The $L$-function values are $L(1,\chi)=\pi/4$ and $L(3)=\pi^3/32$. Thus,
\begin{equation}
\frac{1}{2} =\frac{L(1,\chi )^3}{L(3,\chi)} =\prod_{p \geq 2}\frac{\left (1-\chi(p)p\right ) ^{-3}}{\left (1-\chi(p)p^{-3}\right ) ^{-1}}
=\prod_{p \equiv 1 \bmod 4}\frac{p^3-1}{(p-1)^{3}}\prod_{p \equiv 3 \bmod 4}\frac{p^3+1}{(p+1)^{3}}.
\end{equation}
\end{enumerate}
}
\end{exa}
\section{Partial Products}
\begin{lem} \label{lem9949.82} Fix a real number $s\geq 1$. Let $x \gg 1$ be a large number. Then, 
\begin{enumerate} [font=\normalfont, label=(\roman*)]
\item \begin{equation}
1-\prod_{p > x}\left ( 1 -\frac{1}{p^s} \right )^{-1} \leq \sum_{n >x} \frac{1}{n^s}=\frac{1}{s-1}\frac{1}{x^{s-1}},
\end{equation}
\item \begin{equation}
\zeta(s)- \sum_{n \leq x} \frac{1}{n^s}=O\left (\frac{1}{x^{s-1}}\right ).
\end{equation}
\end{enumerate}
\end{lem}
\begin{proof} These are standard results, see \cite[Theorem 4.11]{TE86}.
\end{proof}

\begin{lem} \label{lem9949.84} Fix a real number $s\geq 1$. Let $x \gg 1$ be a large number. Then, 
\begin{equation}
1-\prod_{p > x}\left ( \frac{p^{2s}-1}{p^{2s}+1} \right )^{-1} \leq \sum_{n >x} \frac{1}{n^{2s}}=\frac{2}{2s-1}\frac{1}{x^{2s-1}} +O\left (\frac{1}{x^{s-1}}\right ).
\end{equation}
\end{lem}
\begin{proof} Routine calculations yield
\begin{eqnarray}
1-\prod_{p > x}\left ( \frac{p^{2s}-1}{p^{2s}+1} \right )^{-1} &=&
1-\prod_{p > x}\left (1+ \frac{2}{p^{2s}-1} \right ) \\
&\leq &\sum_{n >x} \frac{1}{n^{2s}}\nonumber \\
&=&\frac{2}{2s-1}\frac{1}{x^{2s-1}} +O\left (\frac{1}{x^{2s-2}}\right ) \nonumber.
\end{eqnarray}

\end{proof}
\section{Rational Approximations} \label{s1300}
The denominator of the Euler approximation \eqref{eq1300.50} has a trivial bound $q_x \leq( x!)^2$. Sharper and more effective descriptions of the rational approximations generated by some Euler products are provided here.
\begin{lem} \label{lem1300.82} Let $s\geq 1$ be a fixed integer, and let $x \gg 1$ be a large number. Then, the product
\begin{equation}\label{eq1300.50}
\frac{p_x}{q_x}=\prod_{p \leq x}\left ( 1 -\frac{1}{p^s} \right ) ,
\end{equation}
where $\gcd(p_x,q_x)=1$, satisfies the followings properties.
\begin{enumerate} [font=\normalfont, label=(\roman*)]
\item  The integer $p_x\geq 2^{\pi(x)-s-1}$ has exponential growth.
\item The integer $q_x\geq 2^{\pi(x)-s-1}$ has exponential growth.
\end{enumerate}
\end{lem}
\begin{proof} For a large number $x \gg 1$, consider
\begin{equation}
A_x=\prod_{p \leq x}\left ( p^s-1 \right ) \qquad \text{ and } \qquad B_x=\prod_{p \leq x} p^s.
\end{equation}

Here, the integer $A_x$ is divisible by an increasing high power of $2$ as $ x \to \infty$, but the integer $B_x$ is divisible by a small fixed power of $2$:  
\begin{equation}
2^{\pi(x)-1} \mid A_x \qquad \text{ and } \qquad  2^{s} \mid \mid B_x .
\end{equation}
Thus, the even part of the product can be precisely factored as 
\begin{equation} \label{eq1300.80}
\prod_{p \leq x}\left ( \frac{p^s-1}{p^s} \right )=\frac{A_x}{B_x}= \frac{2^{\pi(x)-1}A}{2^{s}B}= \frac{2^{\pi(x)-s-1}A}{B}=\frac{p_x}{q_x},
\end{equation}
where $A>1$ and $B>1$ are integers such that $\gcd(2,B)=1$, and $\gcd(p_x,q_x)=1$. \\

(i) To verify this statement, observe that in \eqref{eq1300.80}, the integer $B$ is odd and that these integers are nearly relatively prime, $1 \leq \gcd(A,B)\leq A$. Hence,
\begin{equation}
p_x =\frac{2^{\pi(x)-s-1}A}{\gcd(A,B)}\geq 2^{\pi(x)-s-1} .
\end{equation}
(ii) To verify this statement, observe that $1/2\leq p_x/q_x\leq1$. Equivalently,
\begin{equation}
\frac{q_x}{2}\leq p_x\leq q_x.
\end{equation}
Hence,
\begin{equation}
q_x \geq p_x \geq 2^{\pi(x)-s-1} .
\end{equation}

These complete the verifications of (i) and (ii).
\end{proof}

\begin{lem} \label{lem1300.84} Let $s\geq 1$ be a fixed integer, and let $x \gg 1$ be a large number. Then, the product
\begin{equation}
\frac{p_x}{q_x}=\prod_{p \leq x} \frac{p^2-1}{p^2+1},
\end{equation}
where $\gcd(p_x,q_x)=1$, satisfies the followings properties.
\begin{enumerate} [font=\normalfont, label=(\roman*)]
\item  The integer $p_x\geq 2^{\pi(x)-s-1}$ has exponential growth.
\item The integer $q_x\geq 2^{\pi(x)-s-1}$ has exponential growth.
\end{enumerate}
\end{lem}
\begin{proof} For a large number $x \gg 1$, consider
\begin{equation}
A_x=\prod_{p \leq x}\left ( p^{2s}-1 \right ) =\prod_{p \leq x}\left ( p^{s}-1 \right )\left ( p^{s}+1 \right )\qquad \text{ and } \qquad B_x=\prod_{p \leq x} \left ( p^{2s}+1 \right ).
\end{equation}

Here, the integer $A_x$ is divisible by an increasing double high power of $2$ as $ x \to \infty$, but the integer $B_x$ is divisible by a high power of $2$:  
\begin{equation}\label{eq1300.85}
2^{2\pi(x)-2} \mid A_x \qquad \text{ and } \qquad  2^{\pi(x)-1} \mid \mid B_x .
\end{equation}
The last expression in \eqref{eq1300.85} follows from $p^{2s}+1\equiv 2 \bmod 4$ for $s \in \N$. Thus, the even part of the product can be precisely factored as 
\begin{equation} \label{eq1300.87}
\prod_{p \leq x}\left (\frac{\left ( p^{s}-1 \right )\left ( p^{s}+1 \right )}{ p^{2s}+1 } \right )=\frac{A_x}{B_x}= \frac{2^{2\pi(x)-2}A}{2^{\pi(x)-1}B}= \frac{2^{\pi(x)-s-1}A}{B}=\frac{p_x}{q_x},
\end{equation}
where $A>1$ and $B>1$ are integers such that $\gcd(2,B)=1$, and $\gcd(p_x,q_x)=1$. \\

(i) To verify this statement, observe that in \eqref{eq1300.87}, the integer $B$ is odd, (follows from $p^{2s}+1\equiv 2 \bmod 4$), and the these integers are nearly relatively prime, $1 \leq \gcd(A,B)\leq A$. Hence,
\begin{equation}
p_x =\frac{2^{\pi(x)-s-1}A}{\gcd(A,B)}\geq 2^{\pi(x)-s-1} .
\end{equation}
(ii) To verify this statement, observe that $1/4\leq p_x/q_x\leq1$. Equivalently,
\begin{equation}
\frac{q_x}{4}\leq p_x\leq q_x.
\end{equation}
Hence,
\begin{equation}
q_x \geq p_x \geq 2^{\pi(x)-s-1} .
\end{equation}

These complete the verifications of (i) and (ii).
\end{proof}
\section{Euclidean Sequences} \label{s2299}
The Euclidean sequence $q_n=\max \{ p \mid n!+1\}$ established the existence of infinitely many primes around 23 centuries ago. The first few terms of the sequence are these: 
\begin{equation}  \label{eq2299.03}
2, \quad 3, \quad 7, \quad 5, \quad 11,\quad 103, \quad 71, \quad 61, \quad 661,\quad 19, \quad 269,\quad 329891, \quad 39916801, \quad 13, \quad 83, \ldots,
\end{equation}
The primes are generated in a chaotic manner. Many variations of this sequence are studied in the literature, see \cite{BA16} and similar references. \\

The Hermite sequence $q_n=\min \{ q \mid (p-1)!+1\}$, where $q\geq 2$ ranges over the primes, generates all the primes numbers, and the primes are generated in increasing order. These nice properties follow from Wilson theorem $(p-1)!+1 \equiv 0 \bmod p$, see \cite[p.\  303]{DL12} for more details. The first few terms of the sequence are these: 
\begin{equation} \label{eq2299.06}
2, \quad 3, \quad 5, \quad 7, \quad 11,\quad 13, \quad 17, \quad 19, \quad 23,\quad 29, \quad 31,\quad 37, \quad 41, \quad 43, \ldots,
\end{equation}
However, the opposite Hermite sequence $q_n=\max \{ q \mid (p-1)!+1\}$  has more complex properties, and generates primes in a chaotic manner.\\

About 2 centuries ago Euler introduced a new primes counting method based on the prime harmonic sum
\begin{equation}
\sum_{p\leq x}\frac{1}{p} \to \log \log x
\end{equation}
as $x \to \infty$, see \cite{EL88}. More recently, about a century ago, Hadamard and delaVallee Poussin independently proved using different methods, that the sequence of increasing prime numbers \eqref{eq2299.06} up to a fixed number $x\ge2 $ has 
\begin{equation}
\pi(x)=\#\{p \leq x\}=\frac{x}{\log x}+O\left ( \frac{x}{\log^2 x}\right )
\end{equation}
primes, confer the literature for additional details.
\section{Irrationality Measures} \label{9900}
The irrationality measure $\mu(\alpha)$ of a real number $\alpha \in \R$ is the infimum of the subset of real numbers $\mu(\alpha)\geq1$ for which the Diophantine inequality
\begin{equation} \label{eq597.36}
  \left | \alpha-\frac{p}{q} \right | \ll\frac{1}{q^{\mu(\alpha)} }
\end{equation}
has finitely many rational solutions $p$ and $q$.

\begin{dfn} \label{dfn9900.01} {\normalfont A measure of irrationality $\mu(\alpha)\geq 2 $ of an irrational real number $\alpha \in \R^{\times}$ is a map $\psi:\N \rightarrow \R$ such that for any $p,q \in \N$ with $q\geq q_0$, 
\begin{equation} \label{eq944.70}
\left | \alpha - \frac{p}{q}  \right | \geq \frac{1}{\psi(q)} .
\end{equation}
Furthermore, any measure of irrationality of an irrational real number satisfies $\psi(q) \geq \sqrt{5}q^{\mu(\alpha)}\geq \sqrt{5}q^2$.
}
\end{dfn}
The concept of measures of irrationality of real numbers is discussed in \cite[p.\ 556]{WM00}, \cite[Chapter 11]{BB87}, et alii.

\begin{lem} {\normalfont (\cite[Theorem 2]{BY08})}\label{lem1200.35} The map $\mu : \mathbb{R} \longrightarrow [2,\infty) \cup \{1\}$ is surjective. Any number in the set $[2, \infty) \cup \{1\}$ is the irrationality measure  of some number. 
\end{lem}

More precisely, 
\begin{enumerate} [font=\normalfont, label=(\arabic*)]
\item  A rational number has an irrationality measure of $\mu(\alpha)=1$, see \cite[Theorem 186]{HW08}.
\item   An algebraic irrational number has an irrationality measure of $\mu(\alpha)=2$, an introduction to the earlier proofs of  Roth Theorem appears in \cite[p.\ 147]{RH94}.
\item   Any irrational number has an irrationality measure of $\mu(\alpha)\geq 2$.
\item   A Mahler number $\psi_b=\sum_{n \geq 1} b^{-[\tau]^ n}$ in base $b\geq 3$ has an irrationality measure of $\mu(\psi_b)=\tau$, for any real number $\tau \geq 2$, see \cite[Theorem 2]{BY08}.
\end{enumerate}

\currfilename.\\


\begin{thebibliography}{99}
\bibitem{BA16} Booker, Andrew R.  \textit{\color{blue} A variant of the Euclid-Mullin sequence containing every prime}. J. Integer Seq. 19 (2016), no. 6, Article 16.6.4, 6 pp.
\bibitem{BB87} Borwein, J. M. and Borwein, P. B.  \textit{\color{blue}AGM: A Study in Analytic Number Theory and Computational Complexity}. New York: Wiley, pp. 362-386, 1987.
\bibitem{BY08} Bugeaud, Yann.  \textit{\color{blue}Diophantine approximation and Cantor sets}. Math. Ann. 341 (2008), no. 3, 677-684.
\bibitem{DL12} De Koninck, Jean-Marie; Luca, Florian. \textit{\color{blue}Analytic number theory. Exploring the anatomy of integers}. Graduate Studies in Mathematics, 134. American Mathematical Society, Providence, RI, 2012.
\bibitem{HW08} G. H. Hardy and E. M. Wright.  \textit{\color{blue}An Introduction to the Theory of Numbers.} 5th ed., Oxford University Press, Oxford, 2008.
\bibitem{BW07} David Burt, Sam Donow, Steven J. Miller, Matthew Schiffman, Ben Wieland. \textit{\color{blue}Irrationality measure and lower bounds for $\pi(x)$}. arXiv:0709.2184.
\bibitem{EL88}  Euler, Leonhard. \textit{\color{blue}Introduction to analysis of the infinite. Book I.} Translated from the Latin and with an introduction by John D. Blanton. Springer-Verlag, New York, 1988.
\bibitem{KE07} E. Kowalski. \textit{\color{blue}Counting Primes Irrationally}. https://people.math.ethz.ch/~kowalski/counting-primes-irrationally.pdf
\bibitem{NW00} Narkiewicz, W.  \textit{\color{blue}The development of prime number theory. From Euclid to Hardy and Littlewood}. Springer Monographs in Mathematics. Springer-Verlag, Berlin, 2000. 
\bibitem{RH94} Rose, H. E.  \textit{\color{blue}A course in number theory.} Second edition. Oxford Science Publications. The Clarendon Press, Oxford University Press, New York, 1994.
\bibitem{RP96} Ribenboim, Paulo.  \textit{\color{blue}The new book of prime number records}. Berlin, New York: Springer-Verlag, 1996.
\bibitem{TE86}Titchmarsh, E. C.  \textit{\color{blue}The theory of the Riemann zeta-function}. Second edition. The Clarendon Press, Oxford University Press, New York, 1986.
\bibitem{WM00} Waldschmidt, Michel. \textit{\color{blue}Diophantine approximation on linear algebraic groups. Transcendence properties of the exponential function in several variables.} Grundlehren der Mathematischen Wissenschaften [Fundamental Principles of Mathematical Sciences], 326. Springer-Verlag, Berlin, 2000.
\end{thebibliography}
\end{document}